\documentclass{article} 
\usepackage{}
\usepackage{amsmath,amssymb,fullpage}
\usepackage{amsthm}
\usepackage{mathtools}
\usepackage{graphicx}
\usepackage{tikz}
\usepackage{dsfont}
\usepackage{amsthm}
\usepackage[english]{babel}
\usepackage{fancyhdr}
\usepackage{wrapfig}
\usepackage{color}
\usepackage[shortlabels]{enumitem}
\usepackage{hyperref}
\usepackage[utf8]{inputenc}
\usepackage{multirow}
\usepackage{comment}
\usetikzlibrary{quotes,angles,positioning}
\usetikzlibrary[arrows.meta,bending]

\newcommand{\Aut}[1]{\text{Aut}(#1)}

\newlist{case}{enumerate}{2}
\setlist[case,1]{label=Case \arabic*:, before=\raggedright}
\setlist[case,2]{label=Subcase \arabic*:, before=\raggedright}

\newtheorem{theorem}{Theorem}

\theoremstyle{definition}
\newtheorem{lemma}[theorem]{Lemma}

\newtheorem{proposition}[theorem]{Proposition}
\newtheorem{conjecture}[theorem]{Conjecture}

\usepackage{authblk}
\author[1]{Alejandra Brewer}
\author[2]{Emma Farnsworth}
\author[3]{Natalie Gomez}
\author[4]{Adam Gregory}
\author[5]{Quindel Jones}
\author[6]{Herlandt Lino}
\author[6]{Darren Narayan}

\affil[1]{Auburn University}
\affil[2]{University of Rochester}
\affil[3]{Texas State University}
\affil[4]{University of Florida}
\affil[5]{Virginia Commonwealth University}
\affil[6]{Rochester Institute of Technology}

\newcommand{\course}{The Asymmetric Index of a Graph}
\newcommand{\assignment}{}

\title{\course\\\assignment}

\date{26 July 2019}

\begin{document}
\maketitle
\begin{abstract}
A graph $G$ is \textit{asymmetric} if its automorphism group is trivial. Asymmetric graphs were introduced by Erd\H{o}s and R\'{e}nyi (1963). They suggested the problem of starting with an asymmetric graph and removing some number $r$ of edges and/or adding some number $s$ of edges so that the resulting graph is non-asymmetric. Erd\H{o}s and R\'{e}nyi defined the \textit{degree of asymmetry} of a graph to be the minimum value of $r+s$. In this paper, we define another property that measures how close a given non-asymmetric graph is to being asymmetric. We define the \textit{asymmetric index} of a graph $G$, denoted $ai(G)$, to be the minimum of $r+s$ so that the resulting graph $G$ is asymmetric. 

\indent We investigate the asymmetric index of both connected and disconnected graphs. We prove that for any non-negative integer $k$, there exists a graph $G$ where $ai(G)=k$. We show that the asymmetric index of a cycle with at least six vertices is two, and provide a complete characterization of all possible pairs of edges that can be added to a cycle to create an asymmetric graph. In addition we determine the asymmetric index of paths, certain circulant graphs, Cartesian products involving paths and cycles, and bounds for complete graphs, and complete bipartite graphs. 
\end{abstract}

\section{Introduction}
	
\hspace{15pt}  We consider undirected graphs without multiple edges or loops. A graph $G$ is \textit{asymmetric} if its automorphism group is trivial. To avoid confusion with \textit{symmetric} graphs where the automorphism group of vertices is all permutations of vertices, a graph with a non-trivial automorphism group of vertices will be referred to as non-asymmetric. Asymmetric graphs were introduced by Erd\H{o}s and R\'{e}nyi \cite{Erdos} in 1963. Any asymmetric graph can be made non-asymmetric by removing some $r$ number of edges and adding some $s$ number of edges. Erd\H{o}s and R\'{e}nyi defined the \textit{degree of asymmetry} $A(G)$ of a graph $G$ to be the minimum of $r+s$. In this paper, we define a property that measures how close a non-asymmetric graph is to being asymmetric. We define the \textit{asymmetric index} of a graph $G$, denoted $ai(G)$, to be the minimum of $r+s$ so that the resulting graph $G$ is asymmetric. At first glance it might appear that calculating the degree of asymmetry of a graph is the inverse problem of calculating the asymmetric index - thinking that adding (removing) edges to (from) an asymmetric graph to obtain a non-asymmetric graph would be the same as removing (adding) edges from (to) a non-asymmetric graph to obtain an asymmetric graph. However, Erd\H{o}s and R\'{e}nyi \cite{Erdos} sought the minimum value of $r+s$ to create symmetry, while in this paper we are seeking the minimum value of $r+s$ to eliminate symmetry. In problems involving the degree of asymmetry, graphs that are far from being asymmetric, such as complete graphs, are not encountered. In fact, if we let $n$ denote the number of vertices in a graph, we will show that when $n\geq 6$, $ai\left( K_{n}\right) \geq \frac{6n}{7}$ but Erd\H{o}s and R\'{e}nyi showed that the degree of asymmetry of a graph with $n$ vertices is less than or equal to $\left\lceil \frac{n-1}{2}\right\rceil $ for all graphs $G$.


For a graph $G$ we will use $V(G)$ to denote the set of vertices, and $E(G)$ to denote the set of edges. The edge between vertices $u$ and $v$ will be denoted $uv$. Two graphs $G$ and $H$ are \textit{isomorphic} if there is a bijection $f:G\rightarrow H$ where $uv\in E(G)\Leftrightarrow f(u)f(v)\in E(H)$. Recall that $f$ is an automorphism if it is an isomorphism from a graph to itself, and the set of all automorphisms of a graph form a permutation group under function composition. We will use $\Aut{G}$ to denote the automorphism group of a graph $G$. The \textit{complement} of a graph $G$ will be denoted $\overline{G}$.  
The \textit{degree} of a vertex $v$ is the number of edges incident to $v$. The \textit{distance} between two vertices $u$ and $v$ is the number of edges in a shortest path between $u$ and $v$ and will be denoted $d(u,v)$. For graphs $G$ and $H$ with disjoint vertex sets, the \textit{join} of $G$ and $H$ is denoted $G\vee H$ and is a graph with the vertices and edges of $G$ and $H$, along with edges between each vertex of $G$ and each vertex of $H$.  Given two graphs $H$ and $K$, with vertex sets $V(H)$ and $V(K)$ the Cartesian product $G=H\Box K$ is a graph where $V(G)=\{(u_{i},v_{j})$ where $u_{i}\in V(H)$ and $v_{j}\in V(K)\}$, and $E(G)=\left\{ (u_{i},v_{j}),(u_{k},v_{l})\right\} $ if and only if $i=k$ and $v_{j}$ and $v_{l}$ are adjacent in $K$ or $j=l$ and $u_{i}$ and $u_{k}$ are adjacent in $H$. For any undefined notation, please see the text \cite{West} by West.

Many papers on asymmetric graphs have followed the seminal paper by Erd\H{o}s and R\'{e}nyi \cite{Erdos}. These include papers by Schweitzer and Schweitzer \cite{Schweitzer}, and L. Quintas \cite{Quintas}. A comprehensive treatment of asymmetric graphs is given in the text by Godsil and Royle \cite{Godsil}.

In this paper we investigate the asymmetric index of a graph for several families of graphs. 
We prove that in some cases vertex-transitive graphs and asymmetric graphs are separated by as few as two edges. We show that the asymmetric index of a cycle with at least six vertices is two, and provide a complete characterization of all possible pairs of edges that can be added to a cycle to create an asymmetric graph. In addition, we obtain the asymmetric index for certain circulant graphs, Cartesian products involving paths and cycles, and bounds for complete graphs, and complete bipartite graphs.

\section{The Index of Asymmetry}

\indent We begin by restating an elementary property regarding asymmetric graphs. 
\begin{proposition} \label{Prop1}
\hypertarget{Prop1}{Given} any graph $G$, $\Aut{G} = \Aut{\overline{G}}$.
\end{proposition}

As a consequence, if $G$ is an asymmetric graph, then the complementary graph $\overline{G}$ is also asymmetric. This leads to the following proposition.

\begin{proposition}
Given any graph $G$, $ai(G) = ai(\overline{G})$.
\end{proposition}

\begin{proof}
Suppose $G$ can be made into an asymmetric graph by removing some set $R$ with $r$ edges and adding some set $S$ with $s$ edges. Then by definition of $\overline{G}$, if we now add those same $r$ edges in $R$ to $\overline{G}$ and remove the same $s$ edges in $S$ from $\overline{G}$, we produce an asymmetric graph.
\end{proof}
We continue by presenting two elementary results involving the join and disjoint union of two non-isomorphic asymmetric graphs.    
    
\begin{proposition}
If $G$ and $H$ are non-isomorphic asymmetric graphs then $G\vee H$ is asymmetric. 
\end{proposition}

\begin{proof}
Since $G$ is an asymmetric graph, and in $G\vee H$, each vertex $u$ in $G$ has the same adjacencies to vertices in $H$, each vertex of $G$ will be unique in $G\vee H$. Similarly each vertex in $H$ will be unique in $G\vee H$.
\end{proof}

\begin{proposition}
If $G$ and $H$ are non-isomorphic asymmetric graphs then $G+H$ is asymmetric.
\end{proposition}

\begin{proof}
By Proposition 1,  if $G$ and $H$ are asymmetric then $\overline{G}$ and $\overline{H}$ are asymmetric. Then by Proposition 3, $\overline{G}$ $\vee $ $\overline{H}$ is asymmetric. Since $\overline{G}$ $\vee $ $\overline{H}=\overline{G+H}$, by Proposition 1, $G+H$ is asymmetric. 
\end{proof}
\hspace{5pt}
Other than the trivial case of a single vertex, it was shown by Erd\H{o}s and R\'{e}nyi \cite{Erdos}
that the next smallest asymmetric graph has six vertices. Hence any graph with five or fewer vertices cannot be made asymmetric by removing or adding edges. 


There is only one asymmetric graph on six vertices \cite{Erdos}, only one asymmetric tree on seven vertices \cite{Quintas}, and only one asymmetric tree on eight vertices. In a graph $G$, two vertices $u$ and $v$ can be \textit{transposed} if there exists an automorphism $\sigma :G\rightarrow G$ in which $\sigma (u)=v$ and  $\sigma (v)=u$. In a graph $G$, a vertex is \textit{unique} if its properties are distinct from the properties of all other vertices in G. In other words, a vertex is unique if and only if it is fixed under every automorphism. Note that a graph $G$ is asymmetric, if all of its vertices are unique.

\begin{lemma}Every asymmetric graph on $n\geq 6$ vertices can be extended to an asymmetric graph on $n+1$ vertices by adding a single vertex and a single edge. 
\end{lemma}

\begin{proof}

Let $G$ be an asymmetric graph without a pendant vertex. Let $G^{\prime }$ be a graph obtained by adding a new vertex $u$ and an edge $uv$, where $v$ is a vertex of maximum degree in $G$. We claim that $G^{\prime }$ is asymmetric. If $G^{\prime }$ is not asymmetric then there exists an automorphism $f$ where two vertices in $G^{\prime }$ can be transposed. We note that any automorphism of $V(G^{\prime })$ must send $v$ to itself since it is the only vertex of degree $\Delta (G)+1$ and $f$ must send $u$ to itself since it is the only vertex of degree $1$. Let $v_{i}$ and $v_{j}$ be two vertices that can be transposed by the automorphism $f$. Since removing the vertex $u$ will impact $v_{i}$ and $v_{j}$ in exactly the same way, then there exists an automorphism of $V(G)$ in which $v_{i}$ and $v_{j}$ could be switched. This would contradict the fact that $G$ is asymmetric.
	
If $G$ has a vertex of degree one, then we choose a vertex $u$ with degree one that has a greatest distance $d$ from a vertex of degree greater than or equal to $3$. Then we can create a new graph $G^{\ast }$ where a vertex $z$ and edge $uz$ are added to $G$. We next show that $G^{\ast }$ is asymmetric. Since $G$ is asymmetric, each in $G-u$ has a property that each of the other vertices does not. The vertex $u$ is the only vertex in $G^{\ast }$ that has degree $2$ and is adjacent to $z$ which has the greatest distance $d+1$ to a vertex of degree of $3$ or more. Hence $u$ and $z$ will both be unique in $G^{\ast }$, making $G^{\ast }$ asymmetric.
\end{proof}

We next present an elementary theorem showing when $ai(G)$ is defined.

\begin{theorem}
The asymmetric index is well-defined for any graph $G$ consisting of a single vertex or having six or more vertices. 
\end{theorem}

\begin{proof} 
Let $G$ be a graph consisting of a single vertex or having six or more vertices. We can simply remove all edges from $G$ and add edges using the construction in the proof of Lemma 5 to create an asymmetric graph on $n$ vertices.
\end{proof}

In the next theorem we show that there exist graphs where the asymmetric index is arbitrarily large.

\begin{theorem}
For any positive integer $k$, there exists a graph $G$ where $ai(G)=k$.
\end{theorem}

\begin{proof}	
We first begin with small cases. Let $T_{7}$ be the asymmetric tree with $7$ vertices. Starting with $8K_{1}$ and adding the six edges of $T_{7}$ shows that $ai(8K_{1})=6$. Successively removing pendant edges will create graphs with asymmetric index $i$ for $2\leq i\leq 5$. It is clear that to obtain asymmetric graphs in each of these cases edges must only be added and adding any smaller positive number of edges will result in a graph with at most five edges which cannot be asymmetric.
	
Starting with $tK_{1}$ where $9\leq t\leq 15$ and first adding the edges of $T_{7}$ and then extending the longest path incident to the vertex of degree $3$ will create graphs with asymmetric index $i$ for $7\leq i\leq 13$. We note that in each of these cases the number of edges in the resulting tree equals the asymmetric index. Starting with $tK_{1}$ and adding fewer than $t-2$ edges will either leave at least two isolated vertices or a component with between $1$ and $5$ edges which cannot be asymmetric.	
It is tempting to continue in this manner, starting with $16K_{1}$ and adding the edges of $T_{7}$ and then extend the longest path incident to a vertex of degree $3$ by eight edges. This will create a graph which shows $ai(16K_{1})\leq 14$. However, it is possible to improve this bound by using the disjoint union of two different non-trivial asymmetric trees. If we start with $16K_{1}$ and add the edges of the asymmetric tree with $7$ vertices, and the edges of the asymmetric tree with $8$ vertices, and leave one isolated vertex, we have an asymmetric graph with $13$ edges (see the figure below). Since the trees have a total of $13$ edges, we have that $ai(15K_{1})\leq 13$. To show equality, we note that $12$ or fewer edges will either be a graph with at least two isolated vertices or a component that has between $1$ and $5$ edges, which cannot be asymmetric.
\begin{figure}[h!]
\center
\includegraphics[scale=0.5]{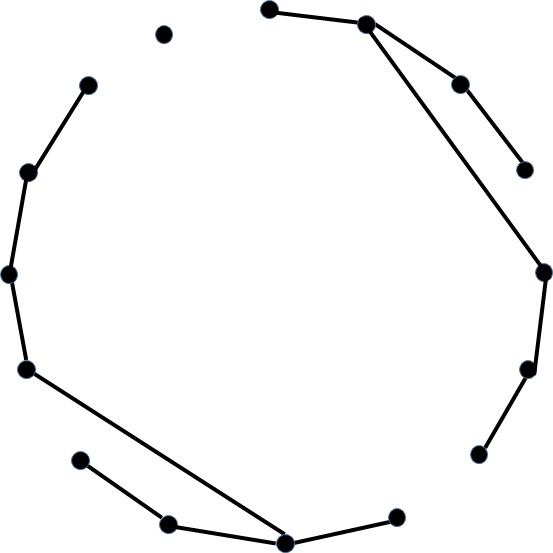}
\caption{An asymmetric graph with 16 vertices}
\end{figure}

To create graphs $G$ with $ai(G)\geq 14$, we use the disjoint union of non-isomorphic asymmetric trees. We first note that $16K_{1}$ is the smallest graph for which we can add the edges of two non-trivial non-isomorphic asymmetric trees and result in an asymmetric graph. Since there is a unique asymmetric tree on $7$ vertices and a unique asymmetric tree on $8$ vertices, the smallest graph for which we can add the edges of three non-trivial non-isomorphic asymmetric trees will contain asymmetric trees on $7$, $8$, and $9$ vertices along with an isolated vertex. This graph will have $25$ vertices and contain $21$ edges. Hence $ai(25K_{1})=21$. Above we showed that $ai(16K_{1})=13$, where we added the edges of $T_{7}$ and $T_{8}$. To create a graphs with asymmetric indices between $14$ and $20$, we start with $tK_{1}$ where $17\leq t\leq 23$ respectively. Then we successively add edges to extend the longest path in $T_{8}$ incident to the vertex of degree $3$ to create trees with $j$ vertices where $9\leq j\leq 15$. In each of these graphs the largest tree will be asymmetric by Proposition 1.

We will use $T_{r}$ to denote a tree with $r$ vertices. Consider the set of non-isomorphic asymmetric trees $T_{a_{1}},T_{a_{2}}, \ldots ,T_{a_{k}}$ where $a_{1}\leq a_{2}\leq \ldots \leq a_{n}$. In all cases $a_{1}=1$ and $a_{2}\geq 6$. There must exist a smallest integer $n_{1}$ such that the edges of the trees $T_{a_{1}},T_{a_{2}},\ldots,T_{a_{k}}$ can be added to $n_{1}K_{1}$ so that the resulting graph $G_{1}$ is asymmetric. This graph will have $\sum_{i=1}^{k} a_{i}$ vertices and $(\sum_{i=1}^{k} a_{i})-k$ edges and hence $ai(G_{1})\leq n_{1}-k$. \ There must also exist a smallest integer $n_{2}$ such that the edges of the trees $T_{a_{1}},T_{a_{2}},\ldots,T_{a_{k}},T_{a_{k+1}}$ can be added to $n_{2}K_{1}$ so that the resulting graph $G_{2}$ is asymmetric. This graph will have $\sum_{i=1}^{k+1} a_{i}$ vertices and $\sum_{i=1}^{k+1} a_{i})-(k+1)$ edges and hence $ai(G_{2})\leq n_{2}-\left( k+1\right) $. We note that in both cases any smaller set of edges will result in a graph with more than one isolated vertex or a tree with fewer than or equal to five edges. Hence $ai(G_{1})=n_{1}-k$ and $ai(G_{2})=n_{2}-\left( k+1\right) $. Next we can construct graphs with asymmetric index $j$ for all $n_{1}-k<j<n_{2}-\left( k+1\right) $ by appending a path with $q$ edges for $1\leq q\leq \left( a_{k+1}\right) -2$, to the longest path incident to a vertex of degree $3$ in $G_{1}$. This completes the proof.
\end{proof}
	
We note that we can also create connected graphs with $ai(G)=k$ for any non-negative integer $k$. For $1\leq k\leq 6$ we use a similar construction to the one described above starting with $K_{7}$ and removing edges of $G_{i}$ $1\leq i\leq 6$. For cases when $k\geq 7$ we start with $K_{n}$ where $n\geq 7$ and remove the edges shown in Figure 2.


\section{Calculating the asymmetric index of a graph}
In this section we investigate the asymmetric index of several families of graphs. We begin with the family of paths. 
\subsection{Paths}
We consider paths, $P_{n}$ where $n\geq 6$. Removing any number of edges will leave a non-asymmetric graph, in the form of shorter paths. We will show that the addition of a single edge can make the resulting graph asymmetric.
\begin{theorem}
For $n \geq 6, \; ai(P_n) = 1$. 
\end{theorem}
\begin{proof}
Consider a path on $n \geq 6$ vertices with consecutive labels $v_1, v_2, \ldots, v_n$. Adding the edge $v_2 v_4$ will produce a cycle with two pendant paths of different lengths, which is asymmetric. This implies that $ai(P_n) = 1$.
\end{proof}
An example of this fact can be seen in Figure 2.
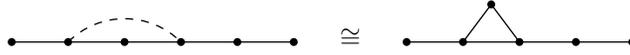
\begin{figure}[h!]
\begin{center}
\begin{tikzpicture}[node distance = 1cm, line width = 0.5pt]
\coordinate (1) at (0,0);
\coordinate (2) at (0.75,0);
\coordinate (3) at (1.5,0);
\coordinate (4) at (2.25,0);
\coordinate (5) at (3,0);
\coordinate (6) at (3.75,0);
\coordinate (7) at (4.5,-0.2);
\node (8) [above = 0.1mm of 7] {$\cong$};
\coordinate (9) at (5.25,0);
\coordinate (10) at (6,0);
\coordinate (11) at (6.375,0.5);
\coordinate (12) at (6.75,0);
\coordinate (13) at (7.5,0);
\coordinate (14) at (8.25,0);

\draw \foreach \x [remember=\x as \lastx (initially 1)] in {2,3,4,5,6}{(\lastx) -- (\x)};
\draw[dashed] (2) to [out=45, in=135] (4);
\draw (9)--(10);
\draw (10)--(11);
\draw (11)--(12);
\draw (12)--(13);
\draw (13)--(14);
\draw (10)--(12);

\foreach \point in {1,2,3,4,5,6,9,10,11,12,13,14} \fill (\point) circle (1.5pt);
\end{tikzpicture}
\end{center}
\caption{Adding a single edge to a path on six vertices}
\end{figure}

We next investigate the different possibilities for a single edge to be added to a path to make the resulting graph asymmetric.

\begin{theorem}
The number of asymmetric graphs obtained by adding an edge to a path $P_{n}$ is
$\left\lfloor \frac{\left( n-4\right) ^{2}}{4}\right\rfloor $.
\end{theorem}

\begin{proof}Let the vertices of $P_{n}$ be $v_{0},v_{1},...,v_{n-1}$. We are interested in adding a single edge to a path so that the resulting graph is a cycle with two pendant paths of different length. This graph is asymmetric since each vertex will have a distinct pair of distances to the vertices of degree $1$.

We first note that adding an edge incident to either $v_{0}$ or $v_{n-1}$ will result in a graph that is not asymmetric. To see this note there will exist a non-trivial automorphism that transposes the two vertices of degree 2 that are adjacent to the vertex of degree 3. For this reason we begin at index 1. Then we note that adding an edge between vertex index $v_i$ and vertex $v_{(n-1)-i}$. Therefore, any vertex $v_i$ cannot connect two at three other vertices  $v_{0}$, $v_{n-1}$, and the vertex $v_{(n-1)-i}$. However, for every one more index we advance farther along the path, we add two more points that cannot be connected to because of the resulting symmetry. For this reason, every vertex on the path has $v-3-2i$ possibilities. Summing over all of the vertex possibilities, we get the total number of possible ways of making the path asymmetric. We only sum up to $\left\lfloor\frac{v}{2} \right\rfloor-2$ because after the halfway point, everything will have been accounted for and thus repeated. Hence the number of possibilities of making this path asymmetric by adding one edge is enumerated by the following formula.
\[\sum\limits_{i=1}^{\left\lfloor\frac{v}{2} \right\rfloor-2}v-3-2i\]

Next we will show that \[\sum\limits_{i=1}^{\left\lfloor\frac{v}{2} \right\rfloor-2}v-3-2i=\left\lfloor \frac{\left( n-4\right) ^{2}}{4}\right\rfloor \]

We first note that when $n$ is odd,
\medskip	

$\sum\limits_{i=1}^{\left\lfloor \frac{n}{2}\right\rfloor -2}\left( n-3-2i\right) $
\medskip	
$=\sum\limits_{i=1}^{\frac{n-5}{2}}\left( n-3-2i\right) \medskip $
\medskip	
$=\allowbreak \frac{1}{4}n^{2}-2n+\frac{15}{4}\medskip $
\medskip	
$=\left\lfloor \frac{\left( n-4\right) ^{2}}{4}\right\rfloor $

When $n$ is even, 
	
\medskip $\sum\limits_{i=1}^{\frac{n}{2}-2}\left( n-3-2i\right) $
\medskip
$=\allowbreak \frac{1}{4}n^{2}-2n+4$\medskip 
\medskip
$=\left\lfloor \frac{\left( n-4\right) ^{2}}{4}\right\rfloor $.

\end{proof}

\subsection{Cycles}

A cycle, $C_n$, is both vertex and edge transitive. Here $\Aut{C_n}$ is far from trivial, in fact it is isomorphic to the dihedral group $D_n$. We first note that $ai(C_n)>1$, since deletion of a single edge will result in a path, which is non-asymmetric and adding a single edge will result in a graph with a reflective line of symmetry that bisects the added edge. 
	
We will show next that for $n \geq 6, \; ai(C_n) = 2$. We consider adding two edges to $C_n$ where $n \geq 6$.
The following three theorems give necessary and sufficient conditions for two edges to be added to a cycle to make the resulting graph asymmetric. There are three ways to add a pair of edges to a cycle to create an asymmetric graph:\ (i) Adding two non-crossing edges that are incident, (ii) two non-crossing edges that are not incident, and (iii) two crossing edges. These three cases are considered in the following three theorems.
	
A vertex $v$ is called distinct if $v$ is fixed under every automorphism of $G$.

\begin{theorem} Let $G$ be the graph where two non-crossing edges are added the cycle $C_{n}$ so that the resulting graph has three chordless cycles $C_{k}$, $C_{m}$, and $C_{l}$, where $C_{m}$ is the subgraph sharing edges with both $C_{k}$ and $C_{l}$ and $k+m+l=n+4$. Then the resulting graph is asymmetric if and only if $k\neq l$.
\end{theorem}

\begin{figure}[h!]
\center
\includegraphics[scale=0.75]{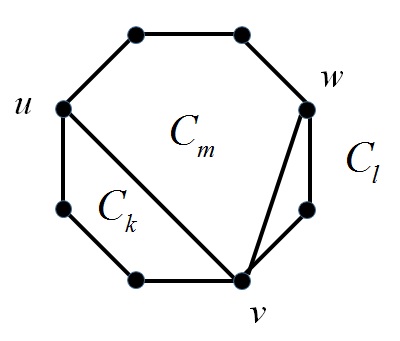}
\end{figure}

\begin{proof}
Assume that $k\neq l$. Let $v$ be the vertex that is incident to both chords, let $u$ be the vertex that is part of both $C_{k}$ and $C_{m}$, and let $w$ be the vertex that part of both $C_{l}$ and $C_{m}$. The vertices $u$,$v$, and $w$ are distinct since they are they only vertices contained in two of the three cycles $C_{k}$, $C_{m}$, and $C_{l}$. Each vertex $x$ in $C_{k}$ is distinct from every vertex in $G$ since the pairs $\left( d(x,u),d(x,v)\right) $ are different for each vertex $x$. Each vertex $y$ in $C_{l}$ is distinct from every vertex in $G$ since the pairs $\left( d(y,u),d(y,w)\right) $ are different for each vertex $y$. Each vertex $z$ in $C_{l}$ is distinct from every vertex in $G$ since the pairs $\left( d(z,v),d(z,w)\right) $ are different for each vertex $y$. For the converse note that if $k=l$ then there is an axis of symmetry that passes through the vertex of degree $4$ and the middle of the arc opposite this vertex. Hence $G$ is non-asymmetric.
\end{proof}

In the following two theorems for vertices $x$ and $y$ on the cycle $C_{n}$ we use $l(x-y)$ to denote the number of edges on the minor arc of the cycle $C_{n}$ between $x$ and $y$.

\begin{theorem}
Let $G$ be the graph where two non-crossing, non-incident edges are added to the cycle $C_{n}$ and the resulting graph has three chordless $C_{k}$, $C_{m}$, and $C_{l}$ , where $C_{m}$ is the subgraph between $C_{k}$ and $C_{l}$ and $k+m+l=n+4$.  Let $t$ and $w$ be the vertices in both $C_{k}$ and $C_{m}$ and let $u$ and $v$ be the vertices in both $C_{m}$ and $C_{l}$.  Then the resulting graph is asymmetric if and only if $k\neq l$ and $l\left( v-w\right) \neq $ $l\left( t-u\right) $. 
\end{theorem}

\begin{figure}[h!]
\center
\includegraphics[scale=0.75]{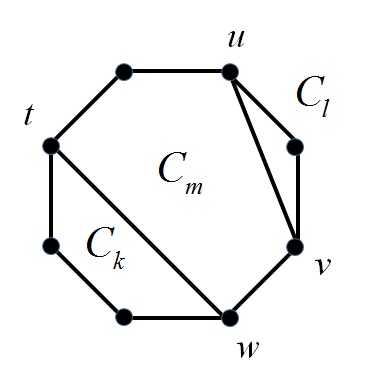}
\end{figure}

\begin{proof}
First assume that $k\neq l$ and $l\left( v-w\right) \neq $ $l\left( t-u\right) $. We will proceed to show that every vertex in $G$ is distinct. We first note that vertices $t$ and $w$ are the only vertices of degree $3$ that are contained in both $C_{k}$ and $C_{m}$. Since $l\left( v-w\right) \neq $ $l\left( t-u\right) $, $t$ and $w$ have different distances to a vertex in $C_{l}$. Hence $t$ and $w$ are distinct. Next note that vertices $u$ and $v$ are the only vertices of degree $3$ that are contained in both $C_{m}$ and $C_{l}$. Since $l\left( v-w\right) \neq $ $l\left( t-u\right) $, $u$ and $v$ have different distances to a vertex in $C_{k}$. Hence $u$ and $v$ are distinct. Then all of the other vertices are distinct since no two vertices on any of the arcs $t-u$, $u-v$, $v-w$, and $w-t$ have the same pair of distances to the end vertices of the arcs they are on. For the reverse direction first note that if $k=l$ then there is an axis of symmetry that passes through the middle of the arc $u-w$ and the middle of the arc $w-t$. If $k\neq l$ and  $l\left( v-w\right) =$ $l\left( t-u\right) $ then $G$ has an axis of symmetry that passes through the middle of the arc $v-w$ and the middle of the arc $t-u$. Hence $G$ is asymmetric.
\end{proof}

\begin{theorem}
Let $G$ be the graph with two crossing chords so that the resulting graph has three chordless $C_{k}$, $C_{m}$, and $C_{l}$ , where $C_{m}$ is the subgraph that share edges with both $C_{k}$ and $C_{l}$ and $k+m+l=n+4$. Furthermore the two edges have vertices $u$ and $v$, and $w$ and $x$, and the resulting graph has four arcs along the cycle, $u-w$, $w-v$, $v-x$, and $x-u$. Then the resulting graph is asymmetric if and only if:
	
(i) $l(v-x)\neq $ $l(x-u)$ or  $l(w-v)\neq $ $l(u-w)$, and
	
(ii) $l\left( u-w\right) \neq l\left( v-x\right) $ and $l\left( w-v\right) \neq l\left( x-u\right) $.
\end{theorem}

\begin{figure}[h!]
\center
\includegraphics[scale=0.75]{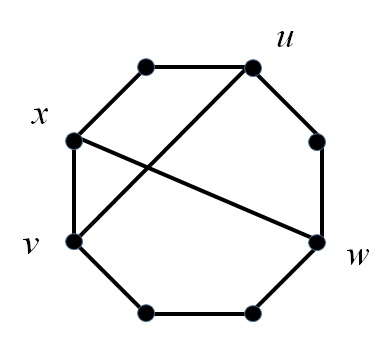}
\end{figure}

\begin{proof}
We first show that if (i) and (ii) both hold the graph is asymmetric. We will show that each vertex is distinct. If (i) holds the graph cannot have any reflectional symmetry about the chord from $x$ to $w$.  Then since $l(u-w)\neq l(v-x)$ and $l(w-v)\neq l(x-u)$ we have that $u,v,w$, and $x$ are all distinct. Then all of the other vertices are distinct since no two vertices on the same arc have the same pair of distances to the two end vertices on the arcs. Hence $G$ is asymmetric.

We next consider if either (i) or (ii) does not hold. If (i) does not hold then $l(v-x)=$ $l(x-u)$ and $l(w-v)=$ $l(u-w)$ then the graph has a line of reflection and is therefore non-asymmetric. Next we consider if (ii) does not hold. If $l\left( u-w\right) =l\left( v-x\right) $, then the graph as an axis of symmetry through the middle of the arcs $x-u$ and $v-w$. The other case involving arcs $w-v$ and $x-u$ is similar.
\end{proof}

Next we will investigate the asymmetric index of wheel graphs.
\begin{theorem}
For $n \geq 6, \; ai(W_n) = 2$. 
\end{theorem}
\begin{proof}
We will first show that $ai(W_n) > 1$. If we remove an edge incident to the vertex of degree $n-1$, this results in a vertex of degree two and the neighbors of this vertex are not unique. Similarly, if we remove an edge whose end points are both degree three, we are left with two vertices of degree two that are each not unique. Now if an edge is added to the graph, it must be added between two vertices of degree three. These vertices now have degree four and are each not unique. It has now been established that $ai(W_n)$ is at least two. 

Now consider the case when a single edge was removed, resulting in two vertices of degree two. If we remove an edge that is incident to one of these vertices of degree two, as well as incident to the vertex of degree $n-1$, then the resulting graph is asymmetric (see Figure 3). 
\end{proof}
\begin{figure}[h!]
\begin{center}
\begin{tikzpicture}[node distance = 0.1cm, line width = 0.5pt]
\coordinate (1) at ({1.5*cos(0)},{1.5*sin(0)});
\coordinate (2) at ({1.5*cos(60)},{1.5*sin(60)});
\coordinate (3) at ({1.5*cos(120)},{1.5*sin(120)});
\coordinate (4) at ({1.5*cos(180)},{1.5*sin(180)});
\coordinate (5) at ({1.5*cos(240)},{1.5*sin(240)});
\coordinate (6) at ({1.5*cos(300)},{1.5*sin(300)});
\coordinate (7) at (0,0);
\coordinate (8) at (2.5,0);
\coordinate (9) at (3,0);

\coordinate (10) at ({5.5+(1.5*cos(0))},{1.5*sin(0)});
\coordinate (11) at ({5.5+(1.5*cos(60))},{1.5*sin(60)});
\coordinate (12) at ({5.5+(1.5*cos(120))},{1.5*sin(120)});
\coordinate (13) at ({5.5+(1.5*cos(180))},{1.5*sin(180)});
\coordinate (14) at ({5.5+(1.5*cos(240))},{1.5*sin(240)});
\coordinate (15) at ({5.5+(1.5*cos(300))},{1.5*sin(300)});
\coordinate (16) at (5.5,0);
\coordinate (24) at (8,0);
\coordinate (25) at (8.5,0);

\coordinate (17) at ({11+1.5*cos(0)},{1.5*sin(0)});
\coordinate (18) at ({11+1.5*cos(60)},{1.5*sin(60)});
\coordinate (19) at ({11+1.5*cos(120)},{1.5*sin(120)});
\coordinate (20) at ({11+1.5*cos(180)},{1.5*sin(180)});
\coordinate (21) at ({11+1.5*cos(240)},{1.5*sin(240)});
\coordinate (22) at ({11+1.5*cos(300)},{1.5*sin(300)});
\coordinate (23) at (11,0);

\draw (1)--(7);
\draw (2)--(7);
\draw (3)--(7);
\draw (4)--(7);
\draw (5)--(7);
\draw (6)--(7);
\draw[->] (8)--(9);

\draw (16)--(11);
\draw[dashed] (10)--(11);
\draw[dashed] (16)--(10);
\draw (10)--(15);
\draw (11)--(12);
\draw (12)--(13);
\draw (13)--(14);
\draw (16)--(15);
\draw (16)--(12);
\draw (16)--(13);
\draw (16)--(14);
\draw (15)--(14);
\draw[->] (24)--(25);

\draw (23)--(18);
\draw (23)--(19);
\draw (23)--(20);
\draw (23)--(21);
\draw (23)--(22);
\draw (18)--(19);
\draw (19)--(20);
\draw (20)--(21);
\draw (21)--(22);
\draw (22)--(17);

\draw \foreach \x [remember=\x as \lastx (initially 1)] in {2,3,4,5,6,1}{(\lastx) -- (\x)};

\foreach \point in {1,2,3,4,5,6,7,10,11,12,13,14,15,16,17,18,19,20,21,22,23} \fill (\point) circle (1.5pt);
\end{tikzpicture}
\end{center}
\caption{Changing a wheel into an asymmetric graph }
\end{figure}
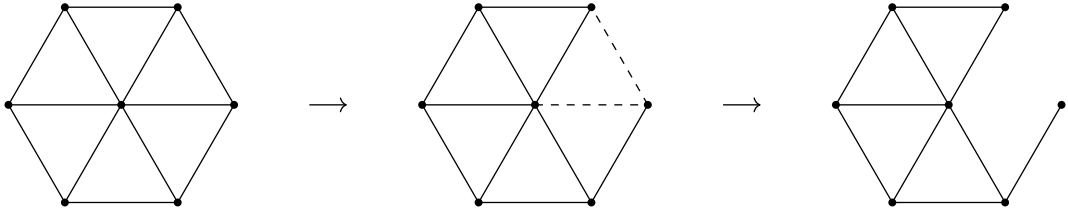

\subsection{Circulant Graphs}

A circulant graph $C_{n}(L)$ is a graph on vertices $v_{1},v_{2},...,v_{n}$ where each $v_{i}$ is adjacent to $v_{i+j \pmod n}$ and $v_{i-j \pmod n}$ for each $j$ in a list $L$.

\begin{theorem} 
Let $G=$ $C_{n}(L)$ where $L$ is a set of $\left\lfloor \frac{n-4}{2}\right\rfloor$ integers between $1$ and $\left\lfloor \frac{n}{2}\right\rfloor$, inclusive. Then $ai\left( G\right)=2$. 
\end{theorem}

\begin{proof}

Let $G=$ $C_{n}(L)$ where $L$ is a set of $\left\lfloor \frac{n-4}{2}\right\rfloor$ integers between $1$ and $\left\lfloor \frac{n}{2}\right\rfloor$, inclusive. The automorphism group of vertices in $C_{n}(L)$ is the dihedral group with $n$ elements (which are rotations or reflections). Since $\left\lfloor \frac{n}{2}\right\rfloor -\left\lfloor \frac{n-4}{2}\right\rfloor \geq 2$ all vertices have degree less than or equal to $n-3$. Adding two edges $xy$ and $xz$ where $y$ and $z$ have different distances from $x$ along the outer cycle will create a graph with no rotations or reflections, which will be asymmetric.
\end{proof}

\begin{figure}[h!]
\center
\includegraphics[scale=0.4]{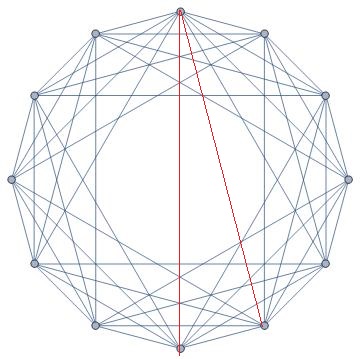}
\caption{An asymmetric graph created by adding two edges to a circulant graph}
\end{figure}

We note that in this theorem the size of $L$ cannot be increased. Suppose $L>\left\lfloor \frac{n-4}{2}\right\rfloor $. If $n$ is odd then the complement of $C_{n}(L)$ will be a cycle and the removal of any two edges will result in a disjoint union of paths, which is not asymmetric. Hence the addition of any two edges to $C_{n}(L)$ will result in a graph that is not asymmetric. If $n$ is even then the complement of $C_{n}(L)$ will be a matching cycle and the removal of any two edges will result in a graph that is not asymmetric. Hence the addition of any two edges to $C_{n}(L)$ will result in a graph that is not asymmetric.
\bigskip
\bigskip
\bigskip	
\bigskip
\begin{theorem}
 Let $n\geq 4$ and let $ML_{2n}$ be the M\"{o}bius Ladder graph with $2n$ vertices. Then $ai\left( ML_{2n}\right) =2$.   
\end{theorem}

\begin{proof}
We note that the complement of $C_{2n}(L)$ where $L$ has $n-2$ elements is the M\"{o}bius Ladder graph with $2n$ vertices. Hence $ai\left( ML_{2n}\right) =ai\left( \overline{C_{2n}(L)}\right) =ai\left( C_{2n}(L)\right) =2$.
\end{proof}

\begin{figure}[h!]
\center
\includegraphics[scale=0.5]{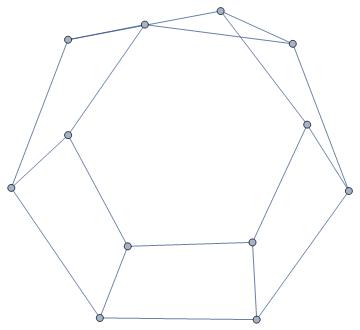}
\caption{M\"{o}bius Ladder graph}
\end{figure}
\subsection{Subdivided stars}

A star with $t$ edges is the complete bipartite graph $K_{1,t}$. As the case where $t=2$ is a path, we will consider $t\geq 3$. A subdivided star $S\left( k_{1},k_{2},...,k_{t}\right) $ is constructed by replacing each of the $t$ edges in $K_{1,t}$ by paths of length $k_{1},k_{2},...,k_{t}$ where each $k_{i}\geq 1$. The vertex to which all pendant paths are incident in the subdivided star is referred to as the \textit{origin vertex}.

\begin{theorem}
A subdivided star $S(k_{1},k_{2},,...,k_{t})$ is asymmetric if and only if $\sum\limits_{i=1}^{t}k_{t}\geq 6$ and $k_{i} \neq k_{j}$ for all $i,j, i \neq j$.
\end{theorem}
\begin{proof}
Let $G=S(k_{1},k_{2},,...,k_{t})$ be a subdivided star. If $G$ fewer than six vertices is is non-asymmetric. If $k_{i} \neq k_{j}$ for every $i,j$ then each vertex is distinct since it can be uniquely characterized by the length of the pendant path containing it and its distance to the central vertex. If there exists $i$ and $j$ where $k_{i}=k_{j}$ then there exists a non-trivial automorphism of $G$ that transposes the vertices at the ends of the pendant paths of length $k_{i}$.

\end{proof}
In Lemma 17 and Theorems 18 and 19 we give necessary and sufficient conditions for a subdivided star to have an asymmetric index of $1$. In Theorem 18 we consider the case of adding an edge and in Theorem 19 we consider the case of removing an edge.

\begin{lemma}
Let $G$ consist of a subdivided star where there are three or fewer pendant paths of the same length incident to a vertex $v$ and an edge $uw$ where $u$ and $w$ are vertices on different pendant paths of the same length and $u$ is adjacent to the origin vertex and $w$ has degree $1$. Then $G$ is asymmetric.
\end{lemma}
\begin{figure}[h!]
\center
\includegraphics[scale=0.75]{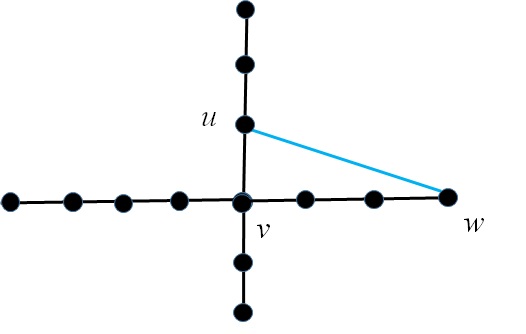}
\caption{A subdivided star with an added edge}
\end{figure}
\begin{proof}
It may be helpful to refer to Figure 6. We will show that each vertex in $G$ is distinct. Clearly $v$ is distinct. The vertex $u$ is distinct since it is the only vertex (other than $v$) that has degree at least $3$. The vertex $w$ is distinct since it is the only vertex of degree $2$ that is part of a cycle and is also adjacent to $u$. All of the other vertices in the cycle have a distinct distance from $w$. The other vertices on the same pendant path as $u$ have distinct distances from $u$. All of the remaining vertices in the graph have a distinct distance from $v$. Hence $G$ is asymmetric.
\end{proof}

By Lemma 17, adding a single edge can make the vertices on two pendant paths of the same length distinct, can make the vertices on three pendant paths of the same length distinct, or can make the vertices on two pendant paths of length $m_{1}$ and two pendant paths of length $m_{2}$ all distinct.

\begin{theorem}
   A subdivided star $G=S(k_{1},k_{2},...,k_{t}) 
   $ where $\sum\limits_{i=1}^{t}k_{i}\geq 7$ can be made asymmetric by adding a single edge if and only if the number of different path lengths is either $t-1$ or $t-2$.
\end{theorem}

\begin{proof}
 If the number of different path lengths is $t-1$ then we have a single pair of pendant paths of the same length $m$, and all of the remaining pendant paths have different lengths that are not equal to $m$. We then add an edge between the two paths between one vertex that is adjacent to the origin vertex and a vertex on the other path that has a degree of $1$. Then since all of the other pendant paths have different lengths, it follows by Lemma 17 that the graph is asymmetric.  If the number of different path lengths is $t-2$ then we have a single triple of pendant paths of the same length $m$, and all of the remaining pendant paths have distinct lengths not equal to $m$, or we have two pairs of pendant paths, $P_{1}$ and $P_{2}$ with length, $m_{1}$ and $Q_{1}$ and $Q_{2}$ with length, $m_{2}$, and all of the remaining pendant paths have distinct lengths not equal to $m_{1}$ or $m_{2}$. We begin with the first case.  If $G$ has three pendant paths of length $m$, let $P$ and $Q$ be two of the pendant paths of length $m$. Let $u$ be the vertex on $P$ that is adjacent to the origin vertex. Let $w$ be the vertex on $Q$ that has degree $1$. Adding the edge $uw$ creates a cycle with $m+2$ vertices. Since all of the remaining pendant paths have different lengths that are not equal to $m$, it follows by Lemma 17 that the graph is asymmetric. For the second case let $u$ be the vertex on $P_{1}$ that is adjacent to the origin vertex. Let $w$ be the vertex on $Q_{1}$ that has degree $1$. Adding the edge $uw$ creates a cycle with $m_{2}+2$ vertices. Then since all of the pendant paths have different lengths, it follows by Lemma 17 that the graph is asymmetric. If the number of distinct lengths of pendant paths is $t$, then the graph is asymmetric by Theorem 16 and hence does not have an asymmetric index of $1$. If the number of distinct lengths is less than or equal to $t-3$, we are guaranteed to have at least five different path lengths distributed among two different pendant paths, or six different path lengths distributed among three different pendant paths. The addition of a single edge will leave two pendant paths with the same length making the graph non-asymmetric. 
\end{proof}

We next consider non-asymmetric subdivided stars that can be made asymmetric by removing a single edge. We first note that removing any edge that is not incident to a vertex of degree $1$ will create a subgraph that is a path on at least two vertices. The resulting graph will be non-asymmetric as there exists an automorphism which transposes the endpoints of the path. Hence the only edges we will consider removing will be incident to a vertex of degree $1$.
	
\begin{theorem}

A non-asymmetric subdivided star $S\left( k_{1},k_{2},...,k_{t}\right) $ where $t\geq 3$ and $\sum\limits_{i=1}^{t}k_{i}\geq 7$ can be made asymmetric by removing an edge if and only if either:
	
\begin{enumerate}[(a)] 
\item $t\geq 4$, $k_{i}=k_{j}=1$ for a pair $i$ and $j$ and $k_{r}\neq k_{s}$ for all $\left\{ r,s\right\} \neq \left\{ i,j\right\} $
	
\item $k_{i}=k_{j}\geq 2$ for a pair $i$ and $j$ and $k_{r}\neq k_{s}$ for all  for all $\left\{ r,s\right\} \neq \left\{ i,j\right\} $, and there does not exist $k_{h}=k_{i}-1$. \end{enumerate}
\end{theorem}
\begin{proof}

First we consider when (a) holds. Assume that $k_{i}=k_{j}=1$ for a pair $i$ and $j$ and $k_{r}\neq k_{s}$ for all other $1\leq r,s\leq t$. Then removal of an edge incident to a vertex of degree $1$ results in a graph that is the disjoint union of a subdivided star where all of the pendant paths have different lengths and an isolated vertex. By Proposition 4 and Theorem 18 this graph is asymmetric.
	
Next we consider when (b) holds. Let $G=S\left( k_{1},k_{2},...,k_{t}\right) $ and there exists a pendant paths with $k_{i}=k_{j}\geq 2$, $k_{r}\neq k_{s}$ for all other $1\leq r,s\leq t$, and no pendant path of length $k_{i}-1$. Then we can remove an edge from either $P_{k_{i}}$ or $P_{k_{j}}$ that is incident to a vertex of degree $1$ and create a graph that is the disjoint union of a subdivided star where all of the pendant paths have different lengths and an isolated vertex. By Proposition 4 and Theorem 16 this graph is asymmetric.
	
For the other direction we consider four cases. 

\begin{enumerate}
    
\item If $t=3$ and $k_{i}=k_{j}=1$ for a pair $i$ and $j$ then removing an edge incident to a vertex of degree $1$ will result in a graph that is the disjoint union of a path and an isolated vertex, which is non asymmetric.
	
\item If $G$ has a single pair pendant paths of of the same length $m\geq 2$ and there is a pendant path of length $m-1$. Removing an edge incident to a vertex of degree $1$ on one of the pendant paths will leave two pendant paths with length $m-1$ so the resulting graph is not asymmetric.
	
\item If $G$ has three pendant paths of of the same length. Removing an edge will still leave two pendant paths of the same length, so the graph is not asymmetric.
	
\item If $G$ has a pair of pendant paths of length $m_{1}$ and a second pair of pendant paths of length $m_{2}$. Then removing one edge will still leave two pendant paths of the same length making the resulting graph non-asymmetric.
\end{enumerate}
\end{proof}




\subsection{Complete graphs}

We next investigate complete graphs $K_{n}$ (and their complements $nK_{1}$ known as discrete graphs) and show they have higher asymmetric indices than other graphs we have encountered. This is expected as the automorphism group of $K_{n}$ is $S_{n}$ and the automorphism group of an asymmetric graph is trivial. We note that $ai\left( K_{1}\right) =0$, and $ai\left( K_{n}\right) $ is not defined when $2\leq n\leq 5$. To determine $ai\left( K_{6}\right)$ we start with six isolated vertices and add the edges of the asymmetric graph on six vertices (shown in Figure 2). To determine $ai\left( K_{7}\right) $ we start with seven isolated vertices and add the edges of the asymmetric graph on six vertices. In both cases using any fewer edges results in a graph that is not asymmetric. Hence $ai\left( 6K_{1}\right) =6$ and $ai\left( 7K_{1}\right) =6$. By Proposition 1, $ai\left( K_{6}\right) =6$ and $ai\left( K_{7}\right) =6$. We have shown in the proof of Theorem 7 when $8\leq n\leq 15$, $ai\left( K_{n}\right) =n-2$. 
It is not too difficult to establish an upper bound for $ai(K_{n})$.  There exists an asymmetric tree $H$ with seven vertices and six edges. By Proposition 4 and Lemma 5 $H$ can be extended to an asymmetric graph $H_{n}$ consisting of a tree with $n-1$ vertices along with an isolated vertex. Then by Proposition 1, $K_{n}-H_{n}$ will be asymmetric. We have shown that for $n\geq 8$,  $ai(K_{n})\leq n-2$. 
For larger cases this bound can be improved, as the exist a larger number of non-isomorphic trees with a total of $n$ vertices. As mentioned in the proof of Theorem 7, $ai\left( K_{16}\right) =13$ since there exist three distinct asymmetric trees with a total of $16$ vertices.

We continue with a larger example. Consider $ai(K_{43})$. It is known that there is a single asymmetric trees with 7 vertices and a single asymmetric tree with 8 vertices and there are three non-isomorphic asymmetric trees on $9$ vertices. By starting with 43 isolated vertices we can add the edges necessary to construct each of these trees. This graph has five non-trivial trees and an isolated vertex, and a total of 37 edges (see Figure 7).

\begin{figure}[h!]
\center
\includegraphics[scale=0.5]{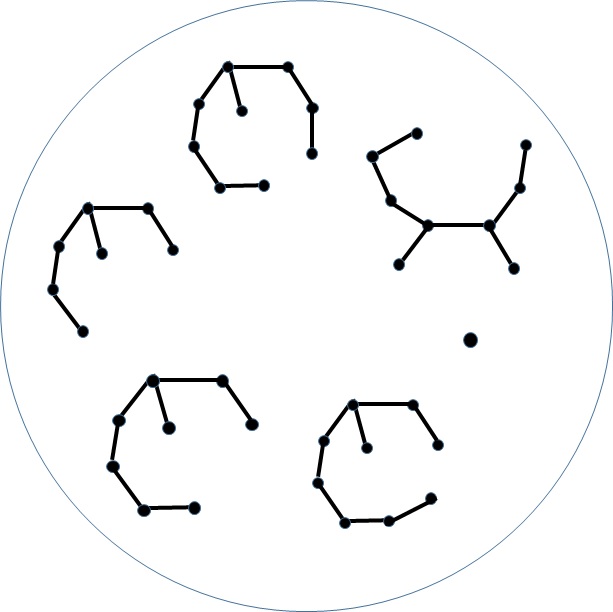}
\caption{An asymmetric graph with 43 vertices and 37 edges}
\end{figure}

Since removing any edge will result in two vertices that are each not unique, it follows that $ai(43K_{1})=37$ which by Proposition 1 implies that $ai(K_{43})=37$. What led to this improvement over $ai(K_{n})=n-2$ is the presence of multiple non-isomorphic asymmetric graphs of the same order. The more of these graphs the better the bound will be.
	 In general when $n\geq 15$, $ai\left( K_{n}\right) =n-t_{n}$ where $t_{n}$ is the number of distinct asymmetric trees that exist on a total of $n$ vertices. We note that in general, $t_{n}$ is not known. In fact, $t_{n}$ relies on the number of asymmetric trees on a fixed number of vertices, which is unknown. T. D. Noe and Alois P. Heinz computed the number of asymmetric trees on $n$ vertices for $n$ = 1...1000 (OEIS A000220 \cite{OEIS}).
We can create a rough general lower bound using multiple copies of the asymmetric tree with seven vertices and six edges. This bound can be improved by taking graphs that are not isomorphic (if they can be identified). Let $G$ be a graph with $n$ vertices. We first isolate a single vertex. Then we construct $\left\lfloor \frac{n-1}{7}\right\rfloor $ sets with seven vertices and one remaining set with $n-1-7\left\lfloor \frac{n-1}{7}\right\rfloor $ vertices. From this we create $\left\lfloor \frac{n-1}{7}\right\rfloor -1$ asymmetric trees with seven vertices and one asymmetric tree with $n-1-7\left( \left\lfloor \frac{n-1}{7}\right\rfloor -1\right) $ vertices. The total number of edges in this graph will be $6\left\lfloor \frac{n-1}{7}\right\rfloor -1+n-1-7\left( \left\lfloor \frac{n-1}{7}\right\rfloor -1\right) -1=\allowbreak n-\left\lfloor \frac{1}{7}n-\frac{1}{7}\right\rfloor +4$
	
Hence we have proved the following general formula which for specific cases can be improved.
\begin{theorem}
For $n\geq 16$, $\allowbreak n-\left\lfloor \frac{1}{7}n-\frac{1}{7}\right\rfloor +4\leq ai(K_{n})\leq n-2$. 
\end{theorem}
    By taking the disjoint union of non-isomorphic asymmetric trees we can construct graphs that have a relatively larger asymmetric index. For example, if we were to take the disjoint union of all non-isomorphic asymmetric trees up to 17 vertices (quantities given by T. D. Noe and Alois P. Heinz \cite{OEIS}) we would create a graph with 43,914 vertices and 41,196 edges. Here $\frac{ai(G)}{n}\approx 0.938$. It appears that by taking the disjoint union of all non-isomorphic asymmetric trees with larger orders, $\frac{ai(G)}{n}\rightarrow 1$.

\subsection{Complete bipartite graphs}
	
	We next present bounds for complete bipartite graphs, $K_{a,b}$.
	
\begin{theorem}	
For $n \geq 6, \; \left\lfloor \frac{n-1}{2}\right\rfloor \leq ai\left( K_{1,n-1}\right) \leq  n-1 $.
\end{theorem}	
\begin{proof}
The lower bound follows by Theorem 9. For the upper bound, note that a path $P_{n-1}$ can be formed using each vertex of degree one, adding $n-2$ edges to the graph. The resulting graph is one edge short of being a wheel graph, and by Theorem 7 we need only remove one additional edge to make the graph asymmetric.  
\end{proof}

\begin{theorem}When $a\geq 2$ and $b\geq 5$, $\left\lfloor \frac{a}{2}\right\rfloor +\left\lfloor \frac{b}{2}\right\rfloor \leq ai\left( K_{a,b}\right) \leq ab-\left( a+b-1\right) $.
\end{theorem}	
	
\begin{proof}Since there exists an asymmetric graph in $t\geq 7$ vertices, we may remove all but $a+b-1$ edges from $K_{a,b}$ and obtain an asymmetric graph. The lower bound follows from Lemma 9.
\end{proof}
\subsection{Cartesian products of paths and cycles}
	We next investigate the asymmetric index for grids and cylinders.
\begin{theorem}For all $r,s\geq 2$, $ai(P_{y}\Box P_{x})=1$.
\end{theorem}

\begin{proof}
All automorphisms of a grid graph $P_{y}\Box P_{x}$ are compositions of horizontal or vertical reflections about the midlines of the grid. As a result in the case where $x$ is odd removing an edge along the middle vertical line will still leave a non-trivial automorpshism which reflects about this vertical line. Removing any other vertical edge will result in an asymmetric graph. Similarly in the case where $y$ is odd removing an edge along the middle horizontal line will still leave a non-trivial automorpshism which reflects about this horizontal line. Removing any other horizontal edge results in an asymmetric graph. If $x$ is even we can remove any vertical edge and obtain an asymmetric graph. If $y$ is even we can remove any horizontal edge and obtain an asymmetric graph.

\end{proof}

\begin{theorem}
For all $p\geq 2$ and $q\geq 3$, $ai(P_{p}\Box C_{q})=2$.
\end{theorem}

\begin{proof}
The only automorphisms of the torus are reflections and rotations.
For the lower bound note that if we remove a single edge $uv$ from $G$ then there is a line of symmetry passing through the missing edge. Hence $ai(P_{p}\Box C_{q})\geq 2$.

\begin{figure}[h!]
\center
\includegraphics[scale=1.0]{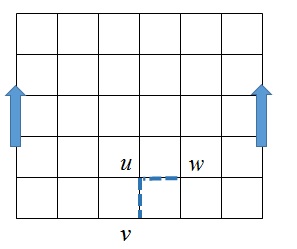}
\caption{A cylinder}
\end{figure}
For the upper bound we remove the edges $uv$ and $uw$ from $P_{p}\Box C_{q}$. This creates a graph without reflectional or rotational symmetries.
\end{proof}

\section{Conclusion}
We have shown that there exist infinite families of graphs with a small asymmetric index, including those with an asymmetric index of $1$. However the largest possible asymmetric index for a graph of order $n$ is not known. We believe that this will occur with complete graphs and for graphs that are the disjoint union of isolated vertices. We formally state this in the following conjecture.

\begin{conjecture} For graphs on $n$ vertices, the graphs with the highest asymmetric index are the complete graph and the discrete graph.
\end{conjecture}

In this paper we studied the asymmetric index of connected graphs. However it would be an interesting problem to investigate the asymmetric index of the graphs that are not connected. Of course, a graph that consists of the disjoint union of non-isomorphic asymmetric graphs is asymmetric. Therefore one approach would be to add or remove edges from the various components to make them asymmetric and so that the components are pairwise non-isomorphic. However there is also the possibility of adding a smaller number of edges that connect two components which may result in an asymmetric subgraph.

\section{Acknowledgements} The authors would like to thank two anonymous referees for a careful reading of our paper and their constructive comments. We would also like thank Rigoberto Fl\'{o}rez and Brendan Rooney for helpful guidance. This research was supported by the National Science Foundation Research for Undergraduates Award 1659075.

\bibliographystyle{plain}
\bibliography{cite5}
\end{document}